\documentclass[reqno,11pt]{amsart}
\usepackage{color}
\usepackage{mathtools}
\usepackage{amsfonts}
\usepackage{amssymb}
\usepackage{amsthm}
\usepackage{newlfont}
\usepackage{graphicx}
\usepackage{float}
\usepackage{amsmath}
\usepackage{geometry}
\usepackage{tikz}
\usetikzlibrary{matrix,shapes,arrows,positioning,chains}

\numberwithin{equation}{section}
\newtheorem{thm}{Theorem}[section]
\newtheorem{lem}{Lemma}[section]
\newtheorem{rem}{Remark}[section]

\newcommand{\ed}{\end {document}}

\begin{document}

\title{The BDF3/EP3 scheme for MBE with no slope selection is stable}


\author[D. Li]{ Dong Li}
\address{D. Li, SUSTech International Center for Mathematics and Department of Mathematics, Southern University of Science and Technology, Shenzhen 518055, P.R. China}
\email{lid@sustech.edu.cn}

\author[C.Y. Quan]{Chaoyu Quan}	
\address{C.Y. Quan, SUSTech International Center for Mathematics, Southern University of Science and Technology,
	Shenzhen 518055, P.R. China}
\email{quancy@sustech.edu.cn}

\author[W. Yang]{ Wen Yang}
\address{\noindent W. ~Yang,~Wuhan Institute of Physics and Mathematics, Innovation Academy for Precision Measurement Science and Technology, Chinese Academy of Sciences, Wuhan 430071, P. R. China.}
\email{wyang@wipm.ac.cn}

\begin{abstract}
We consider the classical molecular beam epitaxy (MBE) model with logarithmic type potential known as no-slope-selection. We employ a third  order backward differentiation (BDF3)  in time with implicit treatment of
the surface diffusion term. The nonlinear term is approximated by a third order explicit extrapolation
(EP3)  formula.  We exhibit mild time step constraints under which the modified energy dissipation
law holds. We  break the second Dahlquist barrier and develop a new theoretical framework to prove unconditional
uniform energy boundedness with no size restrictions on the time step. This is the first unconditional 
result for third order BDF methods applied to the  MBE models without introducing any stabilization term or fictitious variable. The analysis can be generalized to a restrictive class of
phase field models whose nonlinearity has bounded derivatives.
A  novel theoretical framework is also established for the error 
analysis of high order methods.
\end{abstract}

\maketitle

\section{Introduction}
In this work we consider the following molecular beam epitaxy (MBE) model with no slope selection
(cf. \cite{ll2003}):
\begin{align} \label{1.1}
\partial_t h
= -\eta^2 \Delta^2 h   -\nabla \cdot \left(
\frac{\nabla h} {1+|\nabla h|^2} \right), \qquad (t,x) \in (0,\infty) \times \Omega.
\end{align}
Here $\Omega = \mathbb T^2=[-\pi, \pi]^2$ is taken to be the usual two-dimensional
periodic torus and
\begin{align}
|\nabla h |^2 = (\partial_{x_1} h)^2 + (\partial_{x_2} h)^2.
\end{align}
 The function $h=h(t,x): \Omega \to \mathbb R$ represents a scaled height function of the thin film in a co-moving
 frame. The term  $\Delta^2 h$ corresponds to capillarity-driven isotropic surface diffusion (Mullins \cite{m1957}, Herring \cite{h1951}) and the  parameter $\eta^2>0$ is the diffusion coefficient.
 The equation \eqref{1.1} naturally arises from the $L^2$ gradient flow  of the energy functional
\begin{align} \label{1.3}
\mathcal E (h)= \int_{\Omega}
\Bigl( -\frac 12 \log(1+|\nabla h|^2) + \frac 12 \eta^2 | \Delta h |^2 \Bigr) dx.
\end{align}
Due to the negative sign in the logarithmic potential which corresponds to the
Ehrlich-Schwoebel effect, the system often favors uphill atom current and exhibits 
mound-like structures in the film.  If one assumes $|\nabla h |\ll 1$, then the energy
functional \eqref{1.3} can be approximated by
\begin{align} \label{1.3a}
\mathcal E(h) = \int_{\Omega}
\Bigl( \frac 14 (|\nabla h|^2-1)^2 + \frac 12 \eta^2 |\Delta h|^2 \Bigr)dx.
\end{align}
The $L^2$-gradient flow of \eqref{1.3a} leads to the standard MBE model with
slope-selection. The name is derived from the fact that typical solutions corresponding
to \eqref{1.3a} usually ``selects" the slope $|\nabla h| \approx 1$ which exhibits 
pyramidal structures. On the other hand typical solutions to \eqref{1.1} have mound-like
structures and the slopes may have a large upper bound. Analysis-wise these two systems
are vastly different. A well-known difficulty associated with \eqref{1.3a} is the lack of
good a priori Lipschitz bounds (see \cite{lqt2017,lwy2020,lt2021,lqt2016}). In stark contrast the
system \eqref{1.1} has a very benign nonlinearity in the sense that the nonlinear function
$g(z)=-z/(1+|z|^2)$ has bounded derivatives of all orders. This  renders the analysis and simulation quite appealing.

For smooth solutions to \eqref{1.1},  the mean-value of $h$ is preserved in time.
Moreover,   the basic energy conservation law takes the form
\begin{align}
\mathcal E( h(t_2) ) + \int_{t_1}^{t_2} \| \partial_t h \|_2^2 dt = \mathcal
E (h(t_1) ), \qquad\forall\, 0\le t_1 <t_2 <\infty.
\end{align}
This leads to
\begin{align} \label{1.6}
\mathcal E (h(t) ) \le \mathcal E (h(0) ), \qquad\forall\, t>0.
\end{align}
Since the energy is coercive (Lemma \ref{lem_coer}) and the mean-value
of $h$ is preserved, \eqref{1.6} gives global $H^2$ control of the solution. The wellposedness
and regularity of solutions to \eqref{1.1} follows from this and the fact that  the nonlinear function $g(z)=-z/(1+|z|^2)$ have bounded derivatives of all orders.

On the numerical side there is by now a rather extensive literature on designing and analyzing
energy stable numerical schemes for phase field models including Allen-Cahn, Cahn-Hilliard, MBE
and so on, including the convex-splitting schemes \cite{eyre1998unconditionally,wang2010unconditionally,chen2012linear}, 
the implicit-explicit schemes (without stabilization) \cite{lt2021,lt2021-a},
the stabilization schemes  \cite{zhu1999coarsening,xt2006,shen2010numerical,lq2017,SongShu17}, and the scalar auxiliary variable schemes \cite{shen2018scalar,shen2019new}.  A fundamental challenge is
to design fast and accurate, easy to implement and stable numerical schemes for problems possessing
a myriad of temporal and spatial scales. Concerning the epitaxy thin film model, many existing works only with the analysis of  first order and second order in time methods such as first order Backward Differentiation Formula in time with first order extrapolation for the nonlinearity
(BDF1/EP1), second order Backward Differentiation Formula with second order 
extrapolation (BDF2/EP2) and implicit treatment of the surface diffusion term (cf \cite{xt2006,lqt2016} and the references therein). These implicit-explicit (IMEX)
methods are often bundled together with some judiciously chosen stabilization
terms in order to accommodate large time steps and improve energy stability (\cite{xt2006, lqt2016, lq2017-1, lq2017, L21}). Concerning third order accurate schemes for the MBE models, there are very few 
works devoted to the analysis of BDF3 type IMEX methods. In this connection we mention the recent
impressive work of Hao, Huang and Wang (\cite{HHWang}) who considered a BDF3/AB3 discretization scheme with an
additional stabilization term
\begin{align}
-A \Delta t^2 \Delta^2 (h^{n+1}-h^n).
\end{align}
In \cite{HHWang} it was shown that if $A\ge O(\eta^{-2})$ then one can have unconditional energy dissipation
for any time step. We should point out that, whilst embracing additional stabilization terms could
improve the stability of the algorithm, it might introduce unwarranted error terms, which renders
the choice of the stabilization parameter a rather delicate and nontrivial task.  Fine-tuning the form of the stabilization terms is in general a technically demanding task and we refer to the introduction 
of  \cite{lt2021-a} for more in-depth discussions and related bibliography.

The main contribution of this work is as follows.
\begin{enumerate}
\item We consider BDF3/EP3 semi-discretization scheme for the MBE model with no slope selection.
 We quantify explicit and mild time step constraints under which the modified energy dissipation
law holds. 

\item We introduce a new theoretical framework and prove unconditional
uniform energy boundedness \emph{with no size restrictions on the time step}. This is the first unconditional 
result for third order BDF methods applied to the  MBE models without introducing any stabilization terms or fictitious variables. 

\item We develop a novel theoretical framework for the error analysis of BDF3 type methods. This framework
is quite robust and can be generalized to higher order methods. 
\end{enumerate}

Our modest goal is to introduce a new paradigm for the stability and error analysis of phase field models.
In particular for problems whose nonlinearity are sufficiently benign (e.g. having bounded derivatives
of first few orders),   one can establish the following:
\begin{table}
\centering
\renewcommand\arraystretch{1.75}
\begin{tabular}{|l|l|}
\hline
\cline{1-2}
$0<\tau<\infty$ & uniform energy bound \\
\hline
 $0<\tau<\tau_{\mathrm c}$ & energy dissipation \\
\hline
\cline{1-2}
\end{tabular}
\end{table}

In the above $\tau_c$ can be quantified in terms of the parameters of the model under study.
In our MBE model \eqref{1.1}, the optimal $\tau_c= O(\eta^2)$ which is consistent with the typical
temporal-spatial ratio by using dimension analysis.

The rest of this paper is organized as follows. In Section 2 we prove modified energy dissipation
under mild time step constraints. In Section 3 we establish unconditional energy stability which is independent of the time step. In Section 4 we establish the error analysis for the
BDF3/EP3 scheme. In Section 5 we carry
out several numerical simulations. The final section is devoted to concluding remarks.


\section{Energy decay for BDF3/EP3}
We consider the following BDF3/EP3 scheme:
\begin{equation}\label{eq:sch2}
\frac{11h^{n+1}-18h^n+9h^{n-1}-2h^{n-2} }{6\tau}=-\eta^2\Delta^2 h^{n+1}+\nabla \cdot g\left(3\nabla h^n-3\nabla h^{n-1} +\nabla h^{n-2}\right), \quad n\ge 2,
\end{equation}
where
\begin{align}
g(z) = -\frac {z }{1+|z|^2}, \quad z\in \mathbb R^2.
\end{align}

To kick start the scheme we can compute $h^1$ and $h^2$ via a first and second-order
scheme respectively.

\begin{lem} \label{gLip0.1}
Consider  $g(z) =- z /{(1+|z|^2)}$ for $z\in \mathbb R^2$.  We have
\begin{align}
&|g(x)-g(y) | \le  |x-y|, \qquad\forall\, x, y \in \mathbb R^2, \label{Lip0.1a}\\
& x^T (Dg)(z) x \le \frac 18 |x|^2, \qquad\forall\, x, z \in \mathbb R^2. \label{Lip0.1b}
\end{align}
\end{lem}
\begin{rem}
Our proof also extends to general dimensions $d\ge 1$.
\end{rem}
\begin{proof}
Denote $g_1(z)=z/(1+|z|^2)$. 
By using the Fundamental Theorem of Calculus, we have
\begin{align}
g_1(x)-g_1(y) = \int_0^1 (D g_1) (y+\theta(x-y) ) d\theta (x-y).
\end{align}
It suffices for us to examine the spectral norm of the symmetric matrix $(Dg_1)(z)$, where
\begin{align}
(Dg_1)(z) = \frac {\delta_{ij}} {1+|z|^2} - \frac {2 z_i z_j }{ (1+|z|^2)^2}.
\end{align}
Now take any $b\in \mathbb R^d $ with $|b|=1$, and let $b^{\perp}$ be a unit vector
orthogonal to $b$. Clearly 
\begin{align}
z= (z\cdot b) b + (z\cdot b^{\perp} ) b^{\perp}, \qquad |z|^2 =
(z\cdot b)^2 + (z\cdot b^{\perp})^2.
\end{align}
Then
\begin{align}
b^{T} (Dg_1)(z) b & = \frac 1 {1+|z|^2} -\frac {2(b\cdot z)^2} {(1+|z|^2)^2} \le 1.
\end{align}
Also 
\begin{align} \label{Lip0.1cc}
b^{T} (Dg_1)(z) b 
&=\;\frac { 1 } {(1+|z|^2)^2} + \frac { (z\cdot b^{\perp})^2-(b\cdot z)^2} {(1+|z|^2)^2} \notag
 \\
& \ge\;   \frac{1- (b\cdot z)^2} {(1+|z|^2)^2} \ge 
\inf_{s\ge 0} \frac { 1-s}{(1+s)^2} \ge -\frac 18.
\end{align}
It follows that the spectral norm of $Dg$ is bounded by $1$ and \eqref{Lip0.1a} follows easily.
The estimate \eqref{Lip0.1b} follows from \eqref{Lip0.1cc}.
\end{proof}
\begin{lem}[Coercivity of the energy] \label{lem_coer}
Let $\eta>0$. 
For any $h \in H^2(\mathbb T^2)$, we have
\begin{align}
c_1 \|\Delta h\|_2^2 -c_2\le 
\int_{\mathbb T^2} \Bigl( -\frac 12 \log(1+|\nabla h|^2) +
\frac 12 \eta^2 |\Delta h|^2 \Bigr) dx 
\le \frac 12 \eta^2\| \Delta h\|_2^2 ,
\end{align}
where $c_1$, $c_2$  are positive constants depending only on $\eta$.
\end{lem}
\begin{proof}
This follows from the simple observation that
\begin{align}
-\mathrm{const}\cdot( 1+ |\nabla h|) \le -\log(1+ |\nabla h|^2)  \le 0.
\end{align}

\end{proof}

\begin{thm}[Modified energy dissipation]\label{thm5.1}
Consider the scheme \eqref{eq:sch2}. Assume  $h^0$, $h^1$, $h^2 \in H^2(\mathbb T^2)$ and
\begin{align}
0<\tau \le \alpha_1 \eta^2, \qquad\alpha_1=\frac {512}{ 7203}\approx 0.071.
\end{align}
Then
\begin{equation}
\widetilde E_{n+1} \le  \widetilde E_{n}, \quad \forall\, n\ge 2,
\end{equation}
where (below $\delta h^n=h^n-h^{n-1}$)
\begin{align}
\widetilde E_n  &= E_n + \frac 3 {4\tau} \|\delta h^n \|_2^2 + \frac 1 {6\tau}\| \delta h^{n-1} \|_2^2
 + \frac 3 2\| \nabla \delta h^n \|_2^2 + \frac 12 \| \nabla \delta h^{n-1} \|_2^2; \\
 E_n &= \mathcal E (h^n) =
 \int_{\Omega}
\Bigl( -\frac 12 \log(1+|\nabla h^n|^2) + \frac 12 \eta^2 | \Delta h^n |^2 \Bigr) dx.
\end{align}
Furthermore if  for some $\alpha_2>0$,
\begin{align} \label{furt01}
\| \delta h^2 \|_2^2+ \| \delta h^1 \|_2^2 \le \alpha_2 \tau,
\end{align}
then we have the uniform $H^2$ bound:
\begin{align} \label{3.15}
\sup_{n\ge 3} ( \| h^n \|_2 + \| \Delta h^n \|_2) \le \widetilde{C}_1<\infty,
\end{align}
where $\widetilde{C}_1>0$ depends only on ($h^0$, $h^1$, $h^2$, $\eta$, $\alpha_2$). 
\end{thm}
\begin{rem}
The assumption \eqref{furt01} is quite reasonable since typically
$h^1-h^0=O(\tau)$ and $h^2-h^1=O(\tau)$ if we compute $h^1$ and $h^2$ via
a first order scheme such as BDF1/EP1 and a second order scheme such as BDF2/EP2 respectively.
\end{rem}


\begin{proof}

Denote
\begin{align}
\delta h^n = h^n -h^{n-1}.
\end{align}

Taking the $L^2$-inner product with $\delta h^{n+1}$ on both sides of \eqref{eq:sch2}, we obtain
\begin{align}\label{eq:bdf2int}
 & (\frac{11h^{n+1}-18h^n+9h^{n-1}-2h^{n-2} }{6\tau}, \delta h^{n+1})
+\frac 1{2} \eta^2( \| \Delta h^{n+1} \|_2^2 -\| \Delta h^n \|_2^2
 + \| \Delta( \delta h^{n+1})  \|_2^2)  \notag \\
&=\;-(g(\nabla h^n), \nabla (\delta h^{n+1}) )- (g(3\nabla h^n-3\nabla h^{n-1}+\nabla h^{n-2})-g(\nabla h^{n}), \nabla (\delta h^{n+1}) ).
\end{align}

Observe that
\begin{align}  \label{5.9t0}
\frac{11h^{n+1}-18h^n+9h^{n-1} -2h^{n-2} }{6\tau}
& = \frac {11\delta h^{n+1} } {6\tau} - \frac {7\delta h^{n} } {6\tau}+\frac{\delta h^{n-1}}
{3\tau}.
\end{align}

By using \eqref{5.9t0} and the Cauchy-Schwartz inequality, we  have
\begin{equation}
(\frac{11h^{n+1}-18h^n+9h^{n-1} -2h^{n-2} }{6\tau}, \delta h^{n+1})  
\ge \frac {13}{12\tau} \| \delta h^{n+1} \|_2^2
-\frac 7{12\tau} \| \delta h^n\|_2^2 - \frac 1 {6\tau} \| \delta h^{n-1} \|_2^2.
\end{equation}

We set
\begin{align}
F_n=-\frac12\int_{\mathbb{T}^2}\log(1+|\nabla h^n|^2)dx.
\end{align}
Using Taylor expansion and Lemma \ref{gLip0.1}, we obtain
\begin{align}
&F_{n+1} \le F_n + (g(\nabla h^n), \nabla \delta h^{n+1} )
+ \frac 12 \cdot \frac 18 \| \nabla \delta h^{n+1} \|_2^2.
\end{align}
Here Lemma \ref{gLip0.1} is used to control the quadratic term in the Taylor expansion.
This implies
\begin{align}
-(g(\nabla h^n), \nabla \delta h^{n+1} ) \le F_n -F_{n+1}
+ \frac 1 {16} \| \nabla \delta h^{n+1} \|_2^2.
\end{align}

On the other hand by using Lemma \ref{gLip0.1}, we have
\begin{equation}
\begin{aligned}
- (g(3\nabla h^n-3\nabla h^{n-1}+\nabla h^{n-2})-g(\nabla h^{n}), \nabla (\delta h^{n+1}) )& \le 
(2\| \nabla \delta h^n\|_2+\|\nabla \delta h^{n-1} \|_2) \cdot \| \nabla \delta h^{n+1}\|_2.
\end{aligned}
\end{equation}
Collecting the estimates, we have
\begin{align}
& \frac {13}{12\tau} \| \delta h^{n+1} \|_2^2
-\frac 7{12\tau} \| \delta h^n\|_2^2 - \frac 1 {6\tau} \| \delta h^{n-1} \|_2^2+ \eta^2 \frac 12 \|\Delta (\delta h^{n+1} )\|_2^2 \notag \\
\le & E_n-E_{n+1} + \frac {25}{16} \| \nabla (\delta h^{n+1})\|_2^2+
 \| \nabla \delta h^n\|_2^2+\frac 12 \| \nabla \delta h^{n-1}\|_2^2.
\end{align}
Rearranging the terms, we obtain
\begin{align} 
& E_{n+1} + \frac {13}{12\tau} \| \delta h^{n+1} \|_2^2 -\frac {25}{16}
\| \nabla \delta h^{n+1} \|_2^2 +\eta^2 \frac 12 \| \Delta \delta h^{n+1} \|_2^2 \notag \\
\le & \; E_n + \frac 7 {12\tau} \| \delta h^n \|_2^2
+\frac 1 {6\tau} \| \delta h^{n-1} \|_2^2 +
\| \nabla \delta h^n \|_2^2 + \frac 12 \| \nabla \delta h^{n-1} \|_2^2. \label{arr3.0}
\end{align}
Now observe that for $0<\tau \le \frac {512} {7203} \eta^2$, we have
\begin{align}
 & \frac 1{3\tau} \| \delta h^{n+1} \|_2^2 + \eta^2 \frac 12
 \| \Delta \delta h^{n+1} \|_2^2 - \frac {25}{16} \| \nabla \delta h^{n+1} \|_2^2 \notag \\
 \ge & \; \Bigl( \sqrt{ \frac {2\eta^2}{3\tau} } -\frac {25}{16} \Bigr) \| \nabla \delta h^{n+1} \|_2^2
 \ge \frac 32 \| \nabla \delta h^{n+1} \|_2^2. 
 \end{align}
 The decay of the modified energy then follows. The estimate \eqref{3.15}
 follows from \eqref{furt01} and Lemma \ref{lem_coer}.
\end{proof}
\begin{rem}
We explain how to fix the constants in the modified energy. Suppose we want to arrive
at the inequality
\begin{align} \label{arr1}
 & E_{n+1} + \alpha_1 \| \delta h^{n+1} \|_2^2 +
 \alpha_2 \| \delta h^n \|_2^2 +\beta_1 \| \nabla \delta h^{n+1} \|_2^2 +
 \beta_2 \| \nabla \delta h^n \|_2^2 \notag \\
 \le & \; E_n + \alpha_1 \|\delta h^n \|_2^2 + \alpha_2\| \delta h^{n-1} \|_2^2
 + \beta_1 \| \nabla \delta h^n \|_2^2 + \beta_2 \| \nabla \delta h^{n-1} \|_2^2.
 \end{align}
 Then \eqref{arr1} is equivalent to
 \begin{align} \label{arr2}
 & E_{n+1} + \alpha_1 \| \delta h^{n+1} \|_2^2  +\beta_1 \| \nabla \delta h^{n+1} \|_2^2  \notag \\
 \le & \; E_n + (\alpha_1-\alpha_2) \|\delta h^n \|_2^2 + \alpha_2\| \delta h^{n-1} \|_2^2
 + (\beta_1-\beta_2) \| \nabla \delta h^n \|_2^2 + \beta_2 \| \nabla \delta h^{n-1} \|_2^2.
 \end{align}
Matching the RHS of \eqref{arr2} with \eqref{arr3.0}, we obtain
\begin{align}
\alpha_1 = \frac 3{4\tau}, \; \alpha_2 = \frac 1 {6\tau}, \; \beta_1 = \frac 32,\;
\beta_2 =\frac 12.
\end{align}
We then deduce that the LHS of \eqref{arr3.0} must satisfy
\begin{align}
  \frac 1{3\tau} \| \delta h^{n+1} \|_2^2 + \eta^2 \frac 12
 \| \Delta \delta h^{n+1} \|_2^2 - \frac {25}{16} \| \nabla \delta h^{n+1} \|_2^2 
 \ge \frac 32 \| \nabla \delta h^{n+1} \|_2^2. 
 \end{align}
\end{rem}

\section{Uniform boundedness of energy for any $\tau>0$}

\begin{thm}[Uniform boundedness of energy for arbitrary time step]\label{thm6.1}
Consider the scheme \eqref{eq:sch2}. Assume  $h^0$, $h^1$, $h^2 \in H^2(\mathbb T^2)$
satisfies
\begin{align} \label{h1h0}
\int_{\mathbb T^2} h^2 dx =\int_{\mathbb T^2} h^1 dx = \int_{\mathbb T^2} h^0 dx,
\end{align}
and  for some constant $\alpha_2>0$
\begin{align} \label{h1h0.0}
\| \delta h^2 \|_2^2+ \| \delta h^1 \|_2^2 \le \alpha_2 \tau.
\end{align}
Then  for any $\tau>0$, it holds that
\begin{align}
\sup_{n\ge 3} (\| h^n \|_{2} + \| \Delta h^n \|_2 ) \le B_1<\infty,
\end{align}
where $B_1>0$ depends only on ($h^0$, $h^1$, $h^2$, $\eta$, $\alpha_2$).  Note that $B_1$ is independent of $\tau$.

\end{thm}
\begin{rem}
Note that the assumption \eqref{h1h0} is quite reasonable since the mean of $h$ is preserved in time
for the PDE solution. If we compute $h^1$ and $h^2$ using BDF1/EP1 and BDF2/EP2 respectively, then it is easy
to check that \eqref{h1h0} and \eqref{h1h0.0} hold.
\end{rem}
\begin{rem}
To put things into perspective,
it is useful to recall the usual notion of $A$-stability in the classical numerical
ODE textbook (cf. pp. 348 of \cite{SM03}).  Consider 
the family of model ODEs
\begin{align}
y^{\prime} = \lambda y, \qquad \lambda\in \mathbb C, \; \mathrm{Re}(\lambda)<0.
\end{align}
A linear multistep method is absolutely
stable for a given value of $\lambda \tau$ if  each root $z=z(\lambda \tau)$
of the associated stability polynomial satisfies $|z(\lambda \tau)|<1$.  The method is called
$A$-stable if the stability region $\{\lambda\tau: \; \text{the method is absolutely stable for $\lambda \tau$}  \}$ covers the negative complex half-plane.  The notion of $A$-stability is extremely demanding, for example the well-known second Dahlquist barrier (\cite{Dah63}) asserts that:
\begin{enumerate}
\item No explicit linear multistep method is $A$-stable;
\item Implicit methods can have order of convergence at most two;
\item The trapezoidal rule has the smallest error constant $1/12$ amongst all second
order $A$-stable linear multistep methods.
\end{enumerate}
In particular the third order BDF3 method is not $A$-stable. However it was already realized
(cf. pp. 348 of \cite{SM03}) that one can relax the condition of $A$-stability by requiring
that the region of absolute stability should include a large part of the negative half-plane and 
in particular the whole negative real axis. The BDF methods are one of the most efficient
methods in this regard.  By analyzing the characteristic polynomial (cf. pp. 27 of \cite{Ari96}), 
it is known that BDF-$k$ ($k$ denotes the order) methods satisfy the root condition and 
is zero-stable if and only if $k\le 6$ (cf. \cite{c1972,cm1975,f1998}). 
\end{rem}
\begin{rem}
It is possible to reconcile 
the unconditional stability result proved in Theorem \ref{thm6.1} with the classical notion
of stability for ODEs as pointed out in the preceding remark. In the PDE setting here, we only need
the stability region to cover the negative real axis. In yet other words one only need to
demand that the method is absolute stable for the special family:
\begin{align}
y^{\prime} = \lambda y, \qquad \lambda<0.
\end{align}
Since the stability region of BDF3 method covers the negative real analysis, it is natural
to expect  stability for all $\tau>0$. Indeed for small $\tau>0$ the numerical solution is close
to the PDE solution and one should expect energy decay. For $\tau\gtrsim 1$,  the linear dissipation term $-\tau \eta^2 \Delta^2 h^{n+1}$ introduces a nontrivial shift of the stability polynomial. In particular,  all characteristics roots lie
strictly inside the unit disk which make the dynamics very stable. Since our nonlinearity is very benign which can regarded
as an $O(1)$-perturbation at each iterative step, the uniform stability easily follows.
\end{rem}
\begin{proof}
By using \eqref{h1h0} and an induction argument, we have
\begin{align}
\int_{\mathbb T^2} h^n dx = \int_{\mathbb T^2} h^0 dx, \qquad\forall\, n\ge 1.
\end{align}
Denote the average of $h^0$ as $\bar h$ and denote
\begin{align}
y^n = h^n - \bar h.
\end{align}
It is not difficult to check that $y^n$ evolves according to the same scheme \eqref{eq:sch2} where $h^n$ is replaced
by $y^n$.  Thus with no loss we can assume all $h^n$ has mean zero.
Note that we may assume $\tau>\alpha_1 \eta^2$ since the case $0<\tau \le \alpha_1 \eta^2$ is
already covered by Theorem \ref{thm5.1}. With some minor change of notation and relabelling the constants if necessary, our desired result
then follows from Theorem \ref{thm6.1a} below.
\end{proof}

Assume $f^n = (f^n_1, f^n_2)$, $n\ge 1$ is a given sequence of functions on $\mathbb T^2$.
Let $u^n$ evolve according to the scheme:
\begin{align} \label{uS0}
\frac{11u^{n+1}-18u^n+9u^{n-1}-2u^{n-2} }{6\tau} =-\Delta^2 u^{n+1}+ \nabla \cdot f^n, \qquad\, n\ge 2,
\end{align}
We have the following uniform boundedness result.
\begin{thm} \label{thm6.1a}
Consider the scheme \eqref{uS0} with $\tau\ge \tau_0>0$. Assume $u^0$, $u^1$, $u^2\in H^2(\mathbb T^2)$ and have mean zero. Suppose
\begin{align}
\sup_{n\ge 2} \| f^n \|_2 \le  A_0<\infty.
\end{align}
We have
\begin{align}
\sup_{n\ge 3} (\| u^n \|_2 + \| \Delta u^n \|_2) \le  A_1<\infty,
\end{align}
where $A_1>0$ depends only on ($\tau_0$,  $A_0$, $u^0$, $u^1$, $u^2$).
\end{thm}
\begin{proof}
We first rewrite \eqref{uS0} as
\begin{align} \label{6.8t1}
u^{n+1}= 18T_1 u^n - 9T_1 u^{n-1} +2T_1 u^{n-2}+ 6\tau T_1 \nabla \cdot f^n,
\end{align}
where $T_1= (11+6\tau \Delta^2)^{-1}$. One should note that since we are working with mean-zero
functions, we can replace \eqref{6.8t1} by
\begin{align} \label{6.8t0}
u^{n+1}= 18T u^n - 9T u^{n-1} +2T u^{n-2}+ 6\tau T \nabla \cdot f^n,
\end{align}
where 
\begin{align}
\widehat{T}(k) = \frac 1 {11+6\tau |k|^4} \cdot 1_{|k|\ge 1}.
\end{align}
The operator $T$ admits a natural spectral bound, namely
\begin{align}
0<\widehat{T}(k) \le \frac 1 {11+6\tau} \le \frac 1 {11+6\tau_0},\quad
6\tau |k| |\widehat{T}(k)| \le 1, 
 \qquad\forall\, 0\ne k \in \mathbb Z^2.
\end{align}

We now denote
\begin{align}
&Z^{n+1}= (\widehat{u^{n+1}}(k), \widehat{u^n}(k), \widehat{u^{n-1}}(k) )^T, \\
&F^{n+1}=(6\tau \widehat{T}(k) (ik)\cdot \widehat{f^n}(k), \, 0, \, 0)^T, \\
&M=
\begin{pmatrix}
18\widehat{T}(k) & -9 \widehat{T}(k) & 2 \widehat{T}(k) \\
1 & 0 & 0 \\
0 & 1 & 0
\end{pmatrix}.
\end{align}
Clearly
\begin{align}
Z^{n+1} & = M Z^n + F^{n+1} \notag \\
&= M^{n-1}Z^2 + \sum_{j=3}^{n+1} M^{n+1-j} F^j, \qquad\forall\, n\ge 2.
\end{align}
Now for each fixed $k$, by Lemma \ref{lem4.2a}, we have
\begin{align}
|M^{n-1} Z^2| \le |Z^2|, \qquad | M^{n+1-j} F^j |
\le K_1 \rho_1^{n+1-j} |F^j|,
\end{align}
where $K_1>0$ depends only $\tau_0$, and $0<\rho_1<1$ depends only on $\tau_0$. 

We then obtain
\begin{align}
\sup_{n\ge 2} \sup_{0\ne k \in \mathbb Z^2} |Z^{n+1}(k)| \le C_1,
\end{align}
where $C_1$ depends only on ($u^0$, $u^1$, $u^2$, $\tau_0$, $A_0$). Using 
\eqref{6.8t0},  we get
\begin{align}
\sup_{n\ge 2} \sup_{0\ne k \in \mathbb Z^2} ||k|^4 \widehat{u^{n+1}}(k)| \le C_2,
\end{align}
where $C_2$ depends only on ($u^0$, $u^1$, $u^2$, $\tau_0$, $A_0$). 
The desired $H^2$-bound then easily follows.
\end{proof}

\begin{lem} \label{lem4.1}
Let $0<s_0<\frac 1 {11}$. For $0<s\le s_0$ the roots to the equation in $\lambda$ 
\begin{align}
\lambda^3 -18s \lambda^2+9s \lambda -2s =0
\end{align}
are given by 
\begin{align}
&\lambda_1= 6s - \frac {a}{9b} +b; \\
&\lambda_2= 6s+ \frac {1+i\sqrt{3}}{18} \cdot \frac a b - \frac {1-i\sqrt3}2 b; \\
& \lambda_3 = \overline{\lambda_2}= 
6s+ \frac {1-i\sqrt{3}}{18} \cdot \frac a b - \frac {1+i\sqrt3}2 b,
\end{align}
where
\begin{align}
& a=27s-324s^2, \\
& b= \Bigl(s-27s^2 +216s^3 +\sqrt{s^2-27s^3+189s^4} \Bigr)^{\frac 13}.
\end{align}
In particular, we have 
\begin{align}
&2.1 s < \lambda_1(s) \le \lambda_a<1,  \label{T4.21a}\\
& |\lambda_2(s)|=|\lambda_3(s)|  \le \sqrt{\frac 2 {2.1}}<1, \qquad\forall\, 0<s\le s_0,
\label{T4.21b}
\end{align}
where $\lambda_a>0$ depends only on $s_0$.
\end{lem}
\begin{proof}
Since the equation is cubic we have the explicit formula for the roots. It is not difficult to check
that $\lambda_1(s)$ is monotonically increasing in $s$ and $\lambda_1(\frac 1{11})=1$ with
$\lambda_1^{\prime}>0$ for $0<s\le \frac 1{11}$.  The function $\lambda_1(s)-2.1s$ is
also monotonically increasing. Thus \eqref{T4.21a} holds.  The bound \eqref{T4.21b}
follows from the fact that 
\begin{align}
|\lambda_2 \lambda_3|=|\lambda_2|^2 = \frac {2s}{\lambda_1(s)} <\frac 2 {2.1}.
\end{align}
\end{proof}

\begin{lem} \label{lem4.2a}
Let $0<s_0<\frac 1{11}$. Consider the matrix
\begin{align}
M(s) = \begin{pmatrix}
18s   &-9s  & 2s \\
1 &0  &0\\
0 &1 &0
\end{pmatrix},
\end{align}
where $0<s\le s_0$. There exists an integer $n_0\ge 1$ which depends only on $s_0$, such that
\begin{align} \label{3.30a}
\sup_{0<s\le s_0}\sup_{x \in \mathbb R^3, \, |x|=1} | M(s)^{n_0} x| \le \epsilon_0<1,
\end{align}
where $\epsilon_0>0$ depends only on $s_0$. In the above $|x| =\sqrt{x_1^2+x_2^2+x_3^2}$
denotes the usual $l^2$-norm on $\mathbb R^3$.

It follows that 
\begin{align} \label{3.31a}
\sup_{0<s\le s_0}\sup_{x \in \mathbb R^3, \, |x|=1} | M(s)^{n} x| \le K_1 \cdot \rho_1^n,
\qquad\forall\, n\ge 0,
\end{align}
where $0<\rho_1<1$, $K_1>0$ depend only on $s_0$. 
\end{lem}
\begin{rem}
The constraint $s_0<\frac 1{11}$ is absolutely necessary. If $s=\frac 1{11}$, then $M(s)x=x$
for $x=(1, 1,1)^T$.
\end{rem}
\begin{proof}
First we have
\begin{align}
&M(s)^2= \begin{pmatrix}
-9s+324s^2 &2s-162s^2 & 36s^2 \\
18s & -9 s & 2s\\
1 & 0 & 0
\end{pmatrix}, \\
&M(s)^3=s
\begin{pmatrix}
2(1-162s+2916s^2) & 9(13-324s)s & 18s(-1+36s) \\
9(-1+36s) & 2(1-81s) &36s\\
18 & -9 &2
\end{pmatrix}.
\end{align}
Clearly if  $s_1$ is sufficiently small, then we have
\begin{align}
\sup_{0<s\le s_1}\sup_{x \in \mathbb R^3, \, |x|=1} | M(s)^{3} x| \le \frac 12.
\end{align}
We now focus on the regime $s_1\le s\le s_0<\frac  1{11}$.  Consider a fixed $s_*\in [s_1, s_0]$.
By Lemma \ref{lem4.1},  there exists $n_*$ depending on $s_*$ such that
\begin{align}
\sup_{x \in \mathbb R^3, \, |x|=1} | M(s_*)^{n_*} x| \le \epsilon_*<1,
\end{align}
where $\epsilon_*$ also depends on $s_*$. Perturbing around $s_*$, we can find a small neighborhood $J_*$ around $s_*$, such that 
\begin{align}
\sup_{x \in \mathbb R^3, \, |x|=1} | M(s)^{n_*} x| \le \epsilon_1<1,
\qquad\forall\, s \in J_*, 
\end{align}
where $\epsilon_1$ depends only on $s_*$.  The inequality
\eqref{3.30a}  then follows from
a covering argument and the fact that the matrix spectral norm is sub-multiplicative.
The inequality \eqref{3.31a} is a trivial consequence of \eqref{3.30a}.
\end{proof}

\section{Error analysis}
In this section we carry out the error analysis for the BDF3/EP3 scheme. We introduce a new
framework which can be generalized to many other settings especially for higher order methods.

To simplify the presentation, we assume the initial data $h^0 \in H^{m_0}(\mathbb T^2)$, $m_0\ge 20$ and has
mean zero.  The high regularity is mainly needed in the consistency estimate so as to justify
that the PDE solution satisfies the BDF3/EP3 to high precision (see \eqref{consis1}).
The regularity assumption can certainly be lowered but we shall not dwell on this issue 
here. We denote $h$ as the exact PDE solution to the system \eqref{1.1} which clearly
has mean zero and uniform $H^{m_0}$ upper bound for all $t\ge 0$.  To simplify the analysis,
we also assume that 
\begin{align} \label{h1exact}
h^1(x) = h(\tau,x), \qquad h^2(x) = h(2\tau, x), \quad\forall\, x\in \mathbb T^2.
\end{align}
In yet other words, we assume the first two numerical iterates (needed to start the
third order scheme) are computed flawlessly.  This will help to elucidate how errors
are genuinely propagated by the third order scheme whilst all other factors are suppressed.
With some additional minor work we can certainly drop the assumption \eqref{h1exact}
and replaced by some consistency estimates for $h^1-h(\tau)$ and $h^2-h(2\tau)$.
However we shall not pursue this matter here in order to simplify the presentation.

\begin{thm}[Error analysis]\label{thmErr}
Consider the scheme \eqref{eq:sch2}. Assume  $h^0$, $h^1$, $h^2 \in H^{m_0}(\mathbb T^2)$,
$m_0\ge 20$ and 
satisfy \eqref{h1exact}.  Let $T>0$ be given.
 For $\tau>0$ sufficiently small, we have
\begin{align}
\sup_{3\le n \le \frac T {\tau}} \| h^n (\cdot)- h(n\tau,\cdot) \|_2 \le C \cdot \tau^3, 
\end{align}
where $C>0$ is independent of $\tau$.
\end{thm}
\begin{proof}
Throughout this proof we denote by $C_i$ various constants which may depend on
($h^0$, $\eta$, $T$, $m_0$) but do not depend on $\tau$ or $n$.  To ease the notation we shall assume the diffusion coefficient
\begin{align}
\eta =1.
\end{align}
Denote
\begin{align}
\eta^n(x) = h^n(x) - h(n\tau, x), \qquad n\ge 0, \, x\in \mathbb T^2.
\end{align}

Step 1. Uniform $H^{m_0}$ bound. By using Theorem \ref{thm6.1} and a bootstrapping argument,
we have
\begin{align} \label{4.5a}
\sup_{n\ge 3}  \| h^n \|_{H^{m_0}} +\sup_{t\ge 0} \|h(t) \|_{H^{m_0}} \le C_1<\infty. 
\end{align}

Step 2. Consistency. By a simple consistency analysis, we have
\begin{align}
&\frac{11h((n+1)\tau) -18h(n\tau)+9h(({n-1})\tau)-2h(({n-2})\tau) }{6\tau}\notag \\
=&\; -\Delta^2 h((n+1)\tau)+\nabla \cdot \biggl(g\Bigl(3\nabla h(n\tau)-3\nabla h(({n-1})\tau) +\nabla h(({n-2})\tau) \Bigr)\biggr) +e^{n+1}, \qquad n\ge 2,
\end{align}
where (here we need to employ the $H^{m_0}$ regularity estimate)
\begin{align} \label{consis1}
\| e^{n+1} \|_2 \le C_2 \tau^3.
\end{align}
Taking the difference with the corresponding equation for $h^{n+1}$, we obtain
\begin{align} \label{ErN10}
&\frac{11\eta^{n+1}-18\eta^n+9\eta^{n-1}-2\eta^{n-2} }{6\tau} \notag \\
=&-\Delta^2 \eta^{n+1}+\nabla \cdot \Bigl(\alpha_{n,1}\nabla \eta^n+\alpha_{n,2} \nabla \eta^{n-1} +\alpha_{n,3}\nabla \eta^{n-2} \Bigr) +e^{n+1}, \qquad n\ge 2,
\end{align}
where
\begin{align} \label{4.9a}
\sum_{j=1}^3 \| \alpha_{n,j}\|_{H^{m_0-1} }\le C_3.
\end{align}

Step 3. Reformulation.  Since we are working with mean-zero
functions, we can replace \eqref{ErN10} by
\begin{align} 
\eta^{n+1}&= 18T \eta^n - 9T \eta^{n-1} +2T \eta^{n-2}+ 6\tau T e^{n+1} \notag \\
&\quad +6\tau T\nabla \cdot \Bigl(\alpha_{n,1}\nabla \eta^n+\alpha_{n,2} \nabla \eta^{n-1} +\alpha_{n,3}\nabla \eta^{n-2} \Bigr),
\end{align}
where 
\begin{align}
\widehat{T}(k) = \frac 1 {11+6\tau |k|^4} \cdot 1_{|k|\ge 1}.
\end{align}
The operator $T$ admits a natural spectral bound, namely
\begin{align}
0<\widehat{T}(k) \le \frac 1 {11+6\tau}, 
 \qquad\forall\, 0\ne k \in \mathbb Z^2.
\end{align}
Since we shall be working with $L^2$ norm of $\eta^n$ which carries
no derivatives, we rewrite \eqref{6.8t0} as
\begin{align}
\eta^{n+1}&= 18T \eta^n - 9T \eta^{n-1} +2T \eta^{n-2}+ 6\tau T e^{n+1} \notag \\
&\quad +6\tau T\Delta  \Bigl(\alpha_{n,1} \eta^n+\alpha_{n,2} \eta^{n-1} +\alpha_{n,3} \eta^{n-2} \Bigr) \notag \\
& \quad -6\tau T\nabla \cdot \Bigl( \nabla \alpha_{n,1} \eta^n+\nabla \alpha_{n,2}  \eta^{n-1} +\nabla \alpha_{n,3} \eta^{n-2} \Bigr).
\end{align}
We now denote
\begin{align}
&\Phi^{n+1}(k)= (\widehat{\eta^{n+1}}(k), \widehat{\eta^n}(k), \widehat{\eta^{n-1}}(k) )^T, \\
&F^{n+1}(k)=(6\tau \widehat{T}(k) \widehat{e^{n+1}}(k), \, 0, \, 0)^T, \\
&M_1(k)=
\begin{pmatrix}
18\widehat{T}(k) & -9 \widehat{T}(k) & 2 \widehat{T}(k) \\
1 & 0 & 0 \\
0 & 1 & 0
\end{pmatrix}, \\
& G^{n+1}(k)= (-6\tau \widehat{T}(k) |k|^2 \widehat{g_{n+1}}(k), 0, 0)^T, \\
& Z^{n+1}(k) =(-6\tau \widehat{T}(k) ik \cdot \widehat{z_{n+1}}(k), 0, 0)^T,
\end{align}
where
\begin{align}
&g_{n+1} = \alpha_{n,1} \eta^n+\alpha_{n,2} \eta^{n-1} +\alpha_{n,3} \eta^{n-2}, \\
&z_{n+1}=\nabla \alpha_{n,1} \eta^n+\nabla \alpha_{n,2}  \eta^{n-1} +\nabla \alpha_{n,3} \eta^{n-2}.
\end{align}
Clearly
\begin{align}
\Phi^{n+1} & = M_1 \Phi^n + F^{n+1} +G^{n+1} +Z^{n+1}. 
\end{align}

Step 4. Analysis. Let $\epsilon_0>0$ be a small constant. The needed smallness will be specified
later. We discuss two cases.

Case 1: $\tau |k|^4 \ge \epsilon_0$. More precisely we first 
estimate $\Phi^{n+1}(k)$ for $\tau |k|^4 \ge \epsilon_0$.  Denote
\begin{align}
{\Phi_H^{n}}(k) =\Phi^n(k) 1_{|k|\ge (\tau^{-1}\epsilon_0)^{\frac 14}}.
\end{align}
By using \eqref{4.5a} and \eqref{4.9a}, it is not difficult to check that 
\begin{align}
\sup_{|k|\ge (\tau^{-1}\epsilon_0)^{\frac 14}}
|k|^2( |G^{n+1}(k)|+ |Z^{n+1}(k)|) \le \beta_1 \tau^4,
\end{align}
where $\beta_1>0$ depends on $\epsilon_0$.

Then
\begin{align}
\Phi_H^{n+1}(k) = M_1(k)\Phi_H^{n}(k) + F_H^{n+1}(k),
\end{align}
where 
\begin{align}
 \| F_H^{n+1}(k) \|_{l^2_k(0\ne k \in \mathbb Z^2)} \le  (C_4 +C_5\beta_1) \tau^4.
\end{align}
Iterating in $n$, we obtain
\begin{align}
\Phi_H^{n+1}(k)= M_1(k)^{n-1}\Phi_H^2(k) + \sum_{j=3}^{n+1} M_1(k)^{n+1-j} 
F_H^j(k), \qquad\forall\, n\ge 2.
\end{align}
Thanks to the cut-off $\tau |k|^4 \ge \epsilon_0$, we can apply Lemma 
\ref{lem4.2a} to get for each $k$, 
\begin{align}
|\Phi_H^{n+1}(k)| \le K_1 \rho_1^{n-1} |\Phi_H^2(k)|
+\sum_{j=3}^{n+1} K_1 \rho_1^{n+1-j} |F_H^j(k)|, \qquad\forall\, n\ge 2,
\end{align}
where $0<\rho_1<1$, $K_1>0$ depend on $\epsilon_0$.  By \eqref{h1exact} we have
$\Phi_H^2\equiv 0$. It follows that 
\begin{align} \label{4.28a}
\sup_{n\ge 3} \| \Phi_H^n (k) \|_{l_2^k (0\ne k\mathbb Z^2)}
\le \beta_2 \tau^4,
\end{align}
where $\beta_2>0$ depend on $\epsilon_0$.

Case 2: $\tau |k|^4 <\epsilon_0$. We need to  
estimate $\Phi^{n+1}(k)$ for $\tau |k|^4 < \epsilon_0$. Denote
\begin{align}
{\Phi_L^{n}}(k) =\Phi^n(k) 1_{|k|< (\tau^{-1}\epsilon_0)^{\frac 14}}.
\end{align}
By Lemma \ref{lem6:08a}, we write
\begin{align}
M_1(k) = N(s_{\tau,k})^{-1} \Lambda(s_{\tau,k}) N(s_{\tau, k} ),
\end{align}
where 
\begin{align}
s_{\tau,k} = \frac 1 {11+6\tau |k|^4}.
\end{align}
Note that since $|k|\ge 1$, we have $ \tau \le \tau |k|^4 <\epsilon_0$. We shall
take $\epsilon_0$ sufficiently small such that Lemma \ref{lem6:08a} can be applied.
Note that $\epsilon_0$ is an absolute constant.

Denote 
\begin{align}
Y^n(k) = N(s_{\tau, k}) \Phi^n_L(k).
\end{align}
We have
\begin{align} \label{4.33a}
Y^{n+1} (k ) = \Lambda(s_{\tau, k} ) Y^n(k) + F_L^{n+1}(k),
\end{align}
where 
\begin{align}
F_L^{n+1}(k) =N(s_{\tau,k}) (F^{n+1}(k) +G^{n+1}(k) +Z^{n+1}(k))
\cdot 1_{|k|< (\tau^{-1}\epsilon_0)^{\frac 14} }.
\end{align}
Taking the dot product  with $\overline{Y^{n+1}(k)}$ (the complex conjugate
of $Y^{n+1}(k)$) on both sides of  \eqref{4.33a}, summing in $k$ and applying the Cauchy-Schwartz inequality, we obtain
\begin{align} \label{4.35a}
|Y^{n+1}(k)|_{l^2_k}^2
& \le \frac 12 |Y^n(k) |_{l^2_k}^2+
\frac 12 |\Lambda(s_{\tau,k} ) { Y^{n+1}(k)}|_{l^2_k}^2
+B_6\tau^7  \notag \\
& \qquad + B_7 \tau |Y^n(k)|_{l^2_k}^2+\epsilon_1 \tau \left|  { |k|^2}  Y^{n+1}(k)
\right|_{l^2_k}^2,
\end{align}
where $\epsilon_1$ will be taken sufficiently small, and $B_6$, $B_7>0$ depend
on $\epsilon_1$.  Note that to obtain \eqref{4.35a},
we have used the estimate of $\Phi_H^n(k)$ (see \eqref{4.28a}) and also Lemma
\ref{lem6:08a} to bound $N(s_{\tau,k})$.  Also in bounding the term containing $F^{n+1}(k)$,
we used
\begin{align}
  | (N(s_{\tau, k} ) F^{n+1}(k)) \cdot \overline{Y^{n+1}}(k) |_{l_k^1(k \ne 0)}
 &\le  \; \tilde C_1 \tau \|e^{n+1} \|_2 |Y^{n+1}(k) |_{l_k^2(k \ne 0)} \notag \\
 & \le  \; \tau ( \frac {\tilde C_1^2} {\epsilon} \| e^{n+1} \|_2^2
 + \epsilon | Y^{n+1}(k) |_{l_k^2(k\ne 0)}^2) \notag \\
 & \le \; \frac {\tilde C_2} {\epsilon} \tau^7 + \epsilon \tau ||k|^2 {Y^{n+1} }(k)
  |_{l_k^2}^2,
 \end{align}
where $\tilde C_1$, $\tilde C_2$ are constants, and $\epsilon>0$ was chosen sufficiently small.

 By Lemma \ref{lem6:08a} and taking
$\epsilon_1$ to be a sufficiently small absolute constant, we have
\begin{align}
\frac 12 |\Lambda(s_{\tau,k} ) { Y^{n+1}(k)}|_{l^2_k}^2
+\epsilon_1 \tau \left| { |k|^2}  Y^{n+1}(k)
\right|_{l^2_k}^2 \le \frac 12  |Y^{n+1}(k) |_{l_k^2}^2.
\end{align}
It follows that
\begin{align}
|Y^{n+1}(k)|_{l^2_k}^2 \le (1+C_7 \tau) |Y^n(k) |_{l^2_k}^2
+C_8 \tau^7, \quad n\ge 2.
\end{align}
Iterating in $n$ up to $n\le T/\tau$ and noting that $Y^2(k)\equiv 0$, we obtain
\begin{align}
\sup_{3\le n \le T/\tau} |Y^{n+1}(k)_{l^2_k}^2 \le C \cdot \tau^6.
\end{align}
The desired estimate then follows.
\end{proof}

\begin{lem}[Smooth diagonalization of the operator matrix] \label{lem6:08a}
 Consider the matrix
\begin{align}
M(s) = \begin{pmatrix}
18s   &-9s  & 2s \\
1 &0  &0\\
0 &1 &0
\end{pmatrix}.
\end{align}
There exists an absolute constant $\kappa_0>0$ sufficiently small such that if 
$s= \frac 1 {11} (1- \kappa)$ with $0<\kappa \le \kappa_0$, then 
$M(s)$ admits the following diagonalization:
\begin{align}
M(s) = N(s)^{-1} \Lambda(s) N(s),
\end{align}
where $\Lambda(s)=\mathrm{diag}(\lambda_1(s), \lambda_2(s), \lambda_3(s) )$,
and for some absolute constants $B_1>0$, $B_2>0$, 
\begin{align}
& \max\{ |\lambda_1(s)|, \, |\lambda_2(s)|, |\lambda_3(s) |\} \le 1-B_1 \kappa; \\
& \sup_{0<\kappa\le\kappa_0} 
\sup_{x\in \mathbb R^3: |x|=1} (|N(s)^{-1} x| +|N(s) x| ) \le B_2.
\end{align}
\end{lem}
\begin{proof}
Observe that in the limiting case $s=\frac 1{11}$, the matrix $M(\frac 1{11})$
has three eigenvalues given by $1$ and $\frac 1 {22} (7 \pm i\sqrt{39})$.
The result then follows from a simple perturbation argument. One can use the
explicit formula for roots as given in Lemma \ref{lem4.1}.
\end{proof}

\section{Numerical experiments}
In the following numerical experiments, given the initial condition $h^0$, {we employ the second order Runge--Kutta method for computing $h^1$ and the BDF2/EP2 method for computing $h^2$, which ensures the third order  convergence in time.}
The Fourier pseudo-spectral method is used for spatial discretization with $N_x\times N_y$ modes. 
 
\subsection{Comparison with the stabilized scheme}
In this part, we compare the accuracy of the BDF3/EP3 scheme \eqref{eq:sch2} with the stabilized BDF3/EP3 scheme:
\begin{equation}\label{eq:sch3}
\begin{aligned}
&\frac{11h^{n+1}-18h^n+9h^{n-1}-2h^{n-2} }{6\tau}=-\eta^2\Delta^2 h^{n+1} \\
&\qquad\qquad+\nabla \cdot g\left(3\nabla h^n-3\nabla h^{n-1} +\nabla h^{n-2}\right)  - A \tau^2 \Delta^2(h^{n+1}-h^n),
\end{aligned}
\end{equation}
where $A> 0$ is the stabilization parameter and
\begin{align}
g(z) = -\frac {z }{1+|z|^2}, \quad z\in \mathbb R^2.
\end{align}
With similar proof for the BDF3/AB3 scheme in \cite{HHWang}, one can impose some restriction such as 
\begin{equation}\label{eq:restrict}
A\geq \frac{9}{32}\left(\frac{49}{16}\right)^4 \eta^{-2}\approx 24.7398 \eta^{-2}
\end{equation}
for the stabilized BDF3/EP3 scheme \eqref{eq:sch3} to preserve the modified energy dissipation property.

We take the computational domain as the periodic torus $\Omega = [-\pi,\pi]^2$. We take the diffusion parameter $\eta = 1$, the final time $T = 1$, and the number of Fourier modes $N_x\times N_y = 256\times 256$. 
For simplicity, we add a suitable forcing term on the right-hand side of \eqref{1.1}, so that the exact solution is 
\begin{equation}
h_{\mathrm{ext}}(t,x,y) = \cos(t)\sin(x)\sin(y).
\end{equation}
Then, we employ the BDF3/EP3 scheme \eqref{eq:sch2} and the stabilized BDF3/EP3 scheme \eqref{eq:sch3} (adding an implicit forcing term on the right-hand side) respectively to solve the problem. 

The $\ell_2$ and $\ell_\infty$ errors at $T$ are computed for different $\tau$ and $A$, which are illustrated in Figure \ref{fig:error}. 
It is obvious that when $A$ becomes larger, the $\ell_2$ and $\ell_\infty$ errors become larger.
In the case of no stabilization term, i.e., $A=0$, we get the best accuracy. 
This indicates that large stabilization parameter could lead to bad accuracy. 
In particular, in the case of $A = 25$, the energy dissipation law is preserved due to the restriction \eqref{eq:restrict}, but the $\ell_2$ and $\ell_\infty$ errors are hundreds of times larger than the case of no stabilization. 

Moreover, we implement similar experiments for $\eta = 0.5$ and plot the corresponding errors in Figure \ref{fig:error2}.
It can be observed that when we choose $A = 100$ so that \eqref{eq:restrict} is satisfied, the  $\ell_2$ and $\ell_\infty$ errors are still hundreds of times larger than the case of no stabilization.

\begin{figure}[!]
\centering
\includegraphics[width=0.4\textwidth]{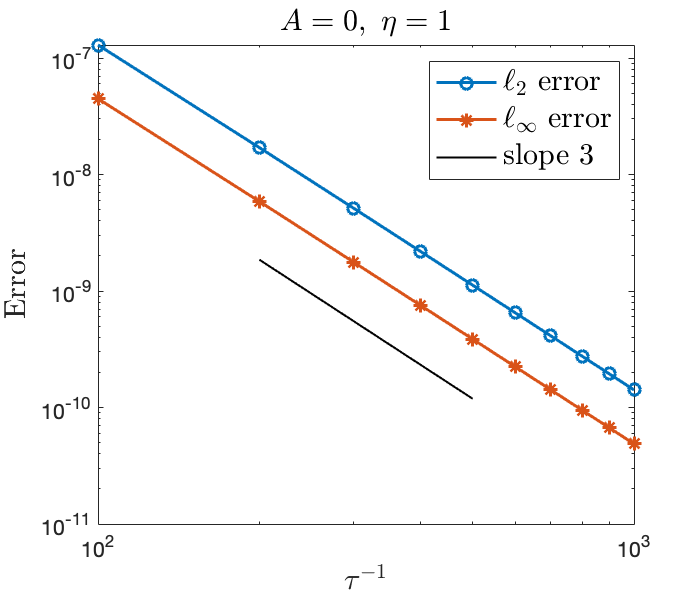}
\includegraphics[width=0.4\textwidth]{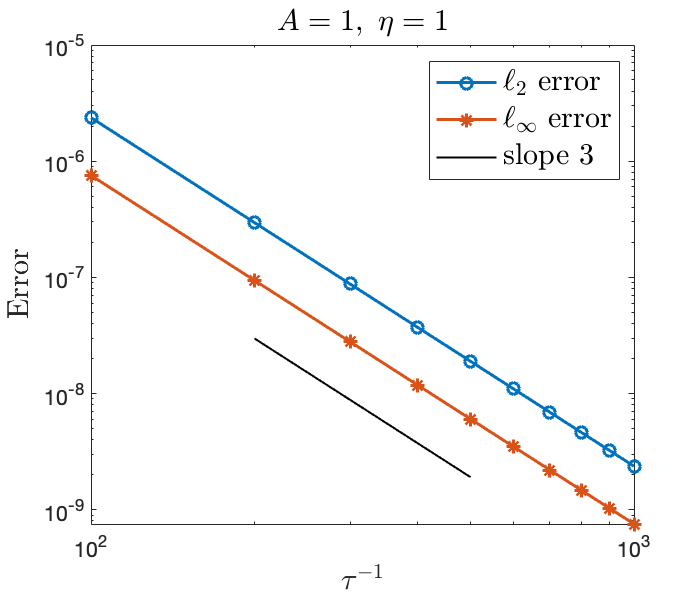}\\
\vspace{0.05in}
\includegraphics[width=0.4\textwidth]{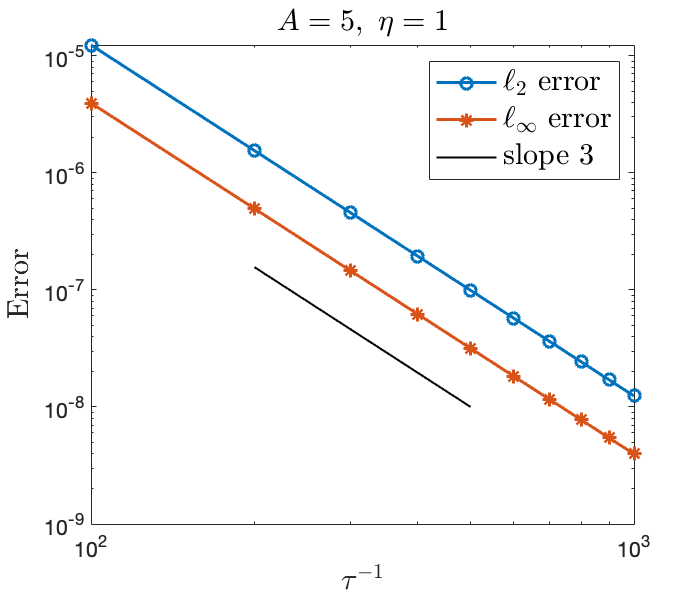}
\includegraphics[width=0.4\textwidth]{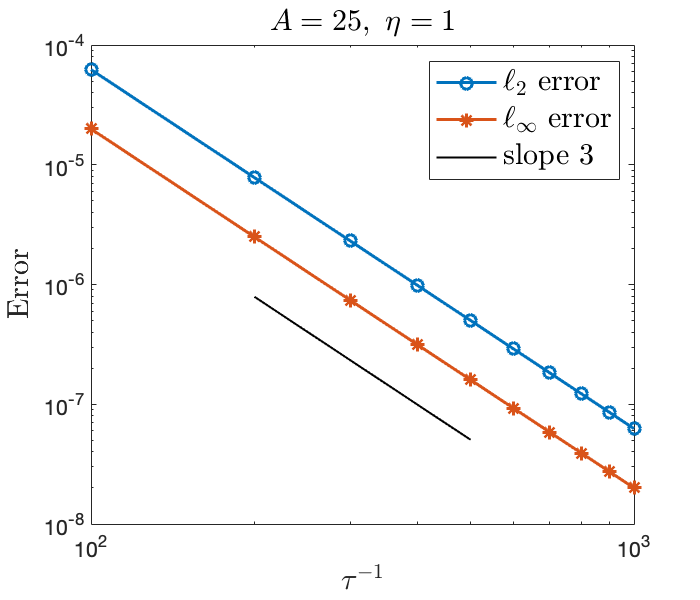}
\caption{\small The $\ell_2$ and $\ell_\infty$ errors at final time $T = 1$ w.r.t. $\tau^{-1}$, computed by the stabilized BDF3/EP3 scheme \eqref{eq:sch3} with $A = 0$ (no stabilization), $1$, $5$, and $25$ respectively, where $\eta = 1,~N_x=N_y = 256$.
 }\label{fig:error}
 \vspace{0.05in}
\centering
\includegraphics[width=0.4\textwidth]{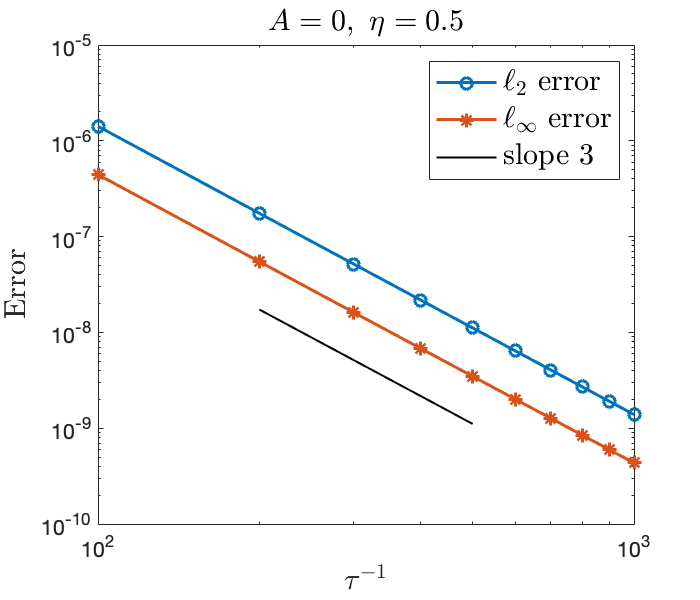}
\includegraphics[width=0.4\textwidth]{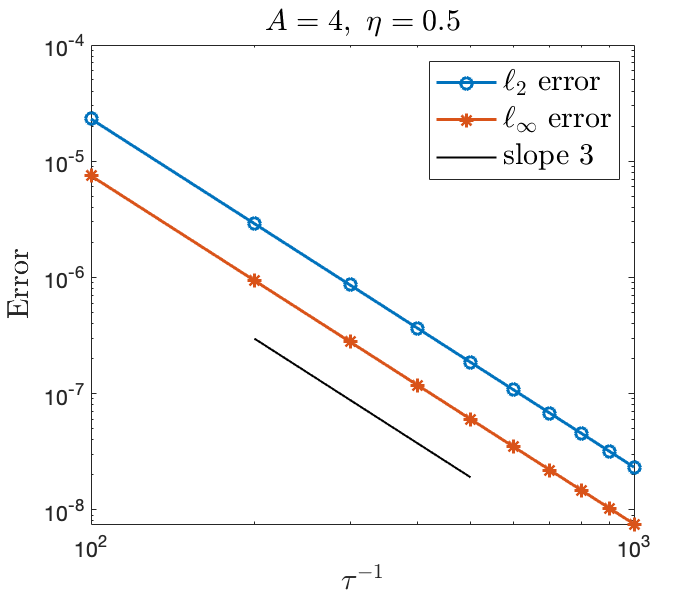}\\
\vspace{0.05in}
\includegraphics[width=0.4\textwidth]{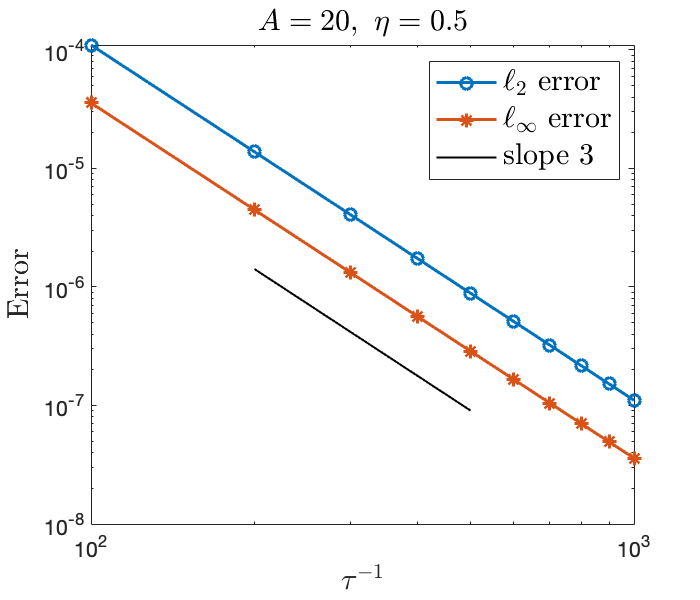}
\includegraphics[width=0.4\textwidth]{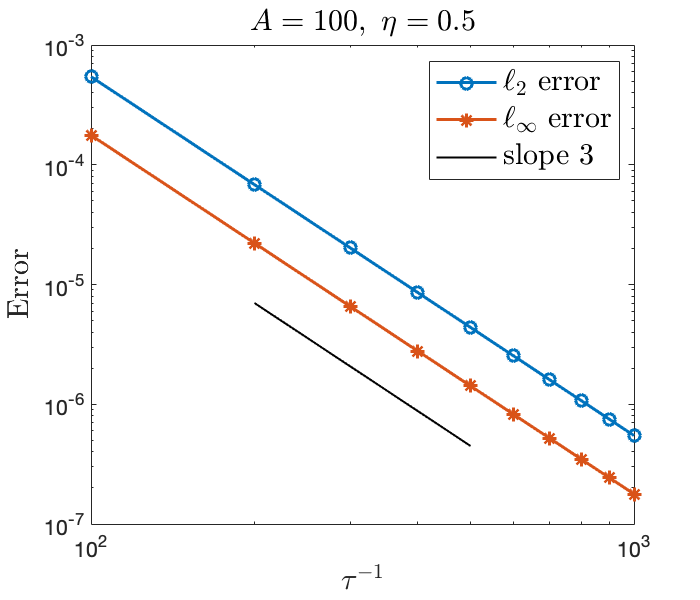}
\caption{\small The same as Figure \ref{fig:error}, except for $\eta = 0.5$ and $A = 0$ (no stabilization), $4,20$, and $100$ respectively.
 }\label{fig:error2}
\end{figure}

It seems that the unconditionally energy dissipation law is too strong a concept since large stabilization parameter might  deteriorate the accuracy. From our analysis the classical IMEX scheme without stabilization might be a better choice. 
The reason is that even if the energy dissipation is not always preserved, the accuracy seems better and the energy is uniformly bounded for any time step $\tau$ as guaranteed by Theorem \ref{thm6.1}. In yet other words, instead of pursuing unconditional energy dissipation, one can try
to accommodate the much weaker notion of unconditional energy stability which seems well suited
for many phase field models. 

{
\subsection{Relation between the standard and the modified energies}
We now clarify the relationship between the standard energy $E_n$ and the modified energy $\widetilde E_n$.
We use the BDF3/EP3 scheme \eqref{eq:sch2} to solve the 2D MBE-NSS equation.
The following parameters are used: $\Omega = [-\pi,\pi]^2$, $\eta = 0.1$, $N_x\times N_y = 256\times 256$, and $h(0,x,y)= \sin(x)\sin(y)$.
In Figure \ref{fig:energy}, the standard energy $E_n$, the modified energy $\widetilde E_n$, and their difference $\Delta E = \widetilde E_n- E_n$ are plotted w.r.t. time.
It can be observed that the standard and the modified energies are approximately the same and nearly
coincide when the time step $\tau$ gets sufficiently small.
}

\begin{figure}[!]
\centering
\includegraphics[width=0.4\textwidth]{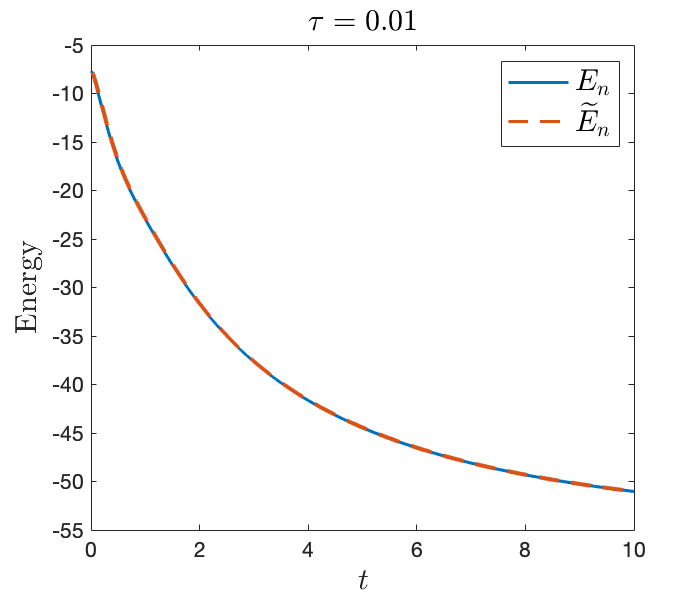}
\includegraphics[width=0.4\textwidth]{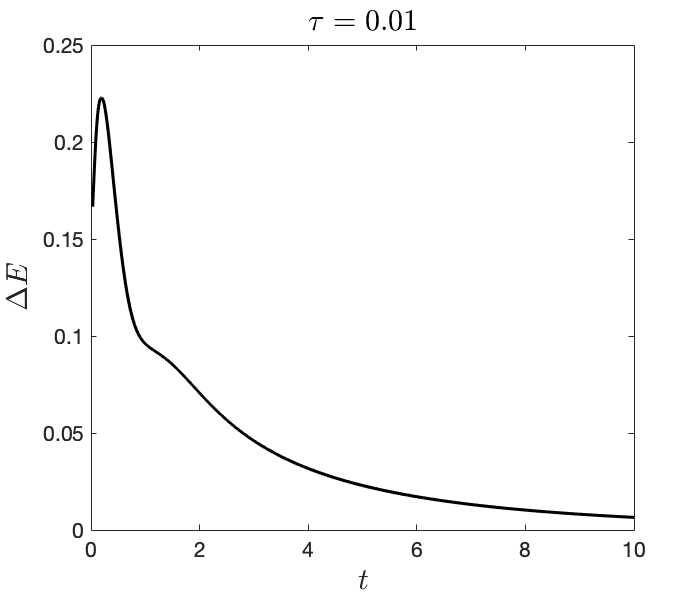}\\
\vspace{0.05in}
\includegraphics[width=0.4\textwidth]{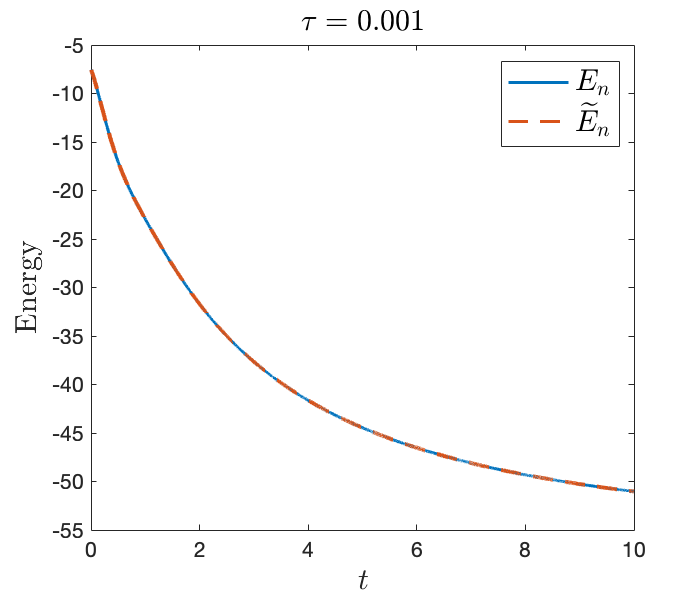}
\includegraphics[width=0.4\textwidth]{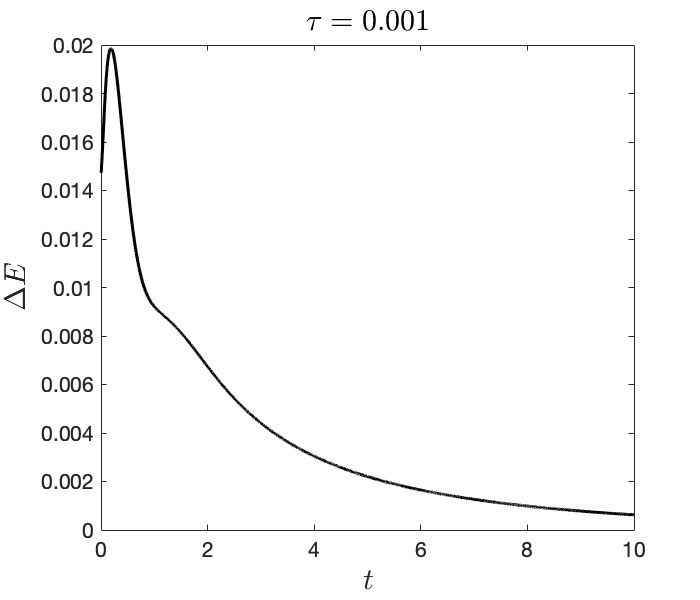}
\caption{\small Standard energy $E_n$, modified energy $\widetilde E_n$, and their difference $\Delta E = \widetilde E_n- E_n$ w.r.t. time, computed by the BDF3/EP3 scheme \eqref{eq:sch2} with $\tau=0.01$ (top) and $0.001$ (bottom). Here, $\eta = 0.1,~N_x=N_y = 256,~h(0,x,y)= \sin(x)\sin(y)$.
 }\label{fig:energy}
 \end{figure}

 { To corroborate our theory, we also test the unconditional energy boundedness for the large time step $\tau=10$.  Figure \ref{fig:energy_largestep} clearly shows that the energy remains bounded
 in time with intermittent small fluctuations violating strict monotonicity. 
}

 \begin{figure}[!]
\centering
\includegraphics[width=0.5\textwidth]{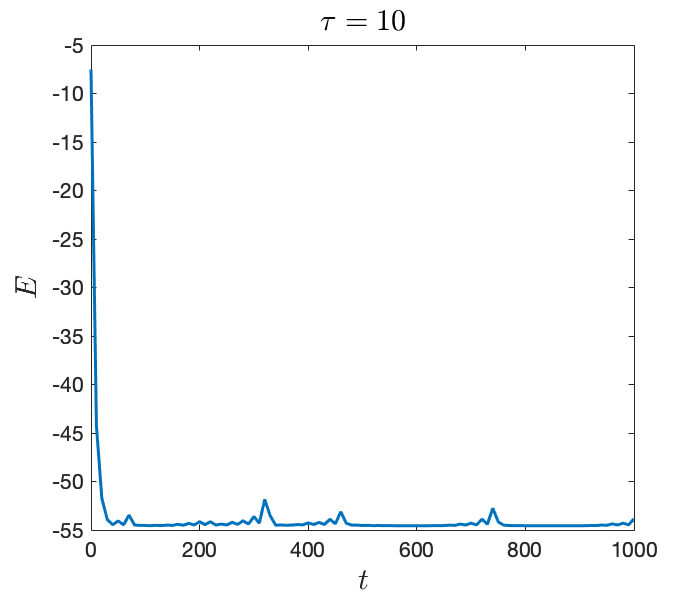}
\caption{\small  Standard energy $E$ w.r.t. time computed with very large time step $\tau= 10$ and other settings the same as in Figure \ref{fig:energy}. }\label{fig:energy_largestep}
\end{figure}

\subsection{Long time simulation}
In this part, we simulate the long time behavior of the coarsening process as described by the thin film model with no slope selection. In the course of simulation we keep track of the evolution of three physical quantities as in \cite{HHWang}: 
\begin{itemize}
\item Energy: 
\begin{align} 
\mathcal E (h)= \int_{\Omega}
\Bigl( -\frac 12 \log(1+|\nabla h|^2) + \frac 12 \eta^2 | \Delta h |^2 \Bigr) \,dx;
\end{align}
\item Characteristic height:
\begin{equation}
H(t) = \frac{1}{\sqrt{|\Omega|}}\|h(t,\cdot)-\bar h(t)\|_2,\quad \mbox{with}\quad \bar h(t) = \frac 1 {|\Omega|} \int_\Omega h(t,x)\,d x;
\end{equation}
\item Characteristic slope:
\begin{equation}
M(t) = \frac{1}{\sqrt{|\Omega|}}\|\nabla h(t,\cdot)\|_2.
\end{equation}
\end{itemize}

We take the computational domain $\Omega = [-\pi,\pi]^2$ (periodic boundary conditions) and 
the final time $T=10^5$.
The initial data is drawn from a uniform distribution in $[0,1]$. 
The following parameters are used: $\eta = 0.01$, $\tau = 0.1$, and $N_x\times N_y = 256\times 256$. 
Note that despite that $\tau = 0.1$ does not satisfies the restriction for energy decay in Theorem \ref{thm5.1}, the energy stability is guaranteed by the uniform boundedness result in Theorem \ref{thm6.1}. 
In Figure \ref{fig:nss_phase}, we illustrate the evolution of $h$ in long time. 
In Figures \ref{fig:nss_E}--\ref{fig:nss_M}, we show the evolutions of $E(t)$, $H(t)$, and $M(t)$, which are fitted respectively as 
\begin{equation}
\begin{aligned}
& E(t) \approx  -9.1691 \log(t)  -53.3853,\\
& H(t) \approx 0.4994 t^{0.4703},\\
& M(t) \approx 6.4993 t^{0.2405}.
\end{aligned}
\end{equation}
As stated in \cite{HHWang}, the lower bound for the energy decay rate is of order $-\log(t)$, and the upper bounds for the evolution rate of average height and average slope are of order $t^{1/2}$, $t^{1/4}$, respectively.
These are consistent with our numerical observations.

\begin{figure}[!]
\centering
\includegraphics[width=0.32\textwidth]{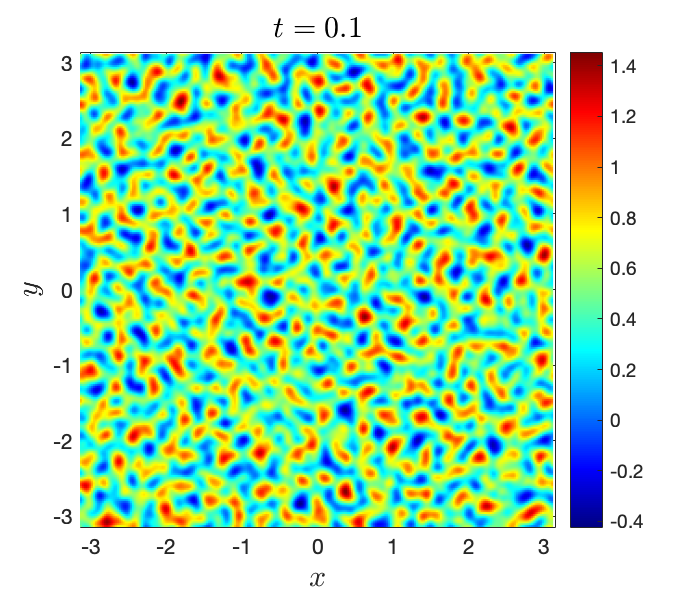}
\includegraphics[width=0.32\textwidth]{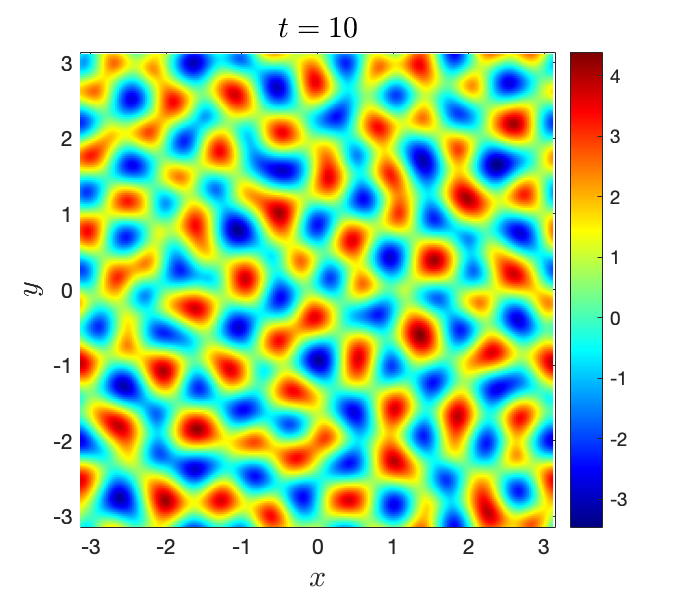}
\includegraphics[width=0.32\textwidth]{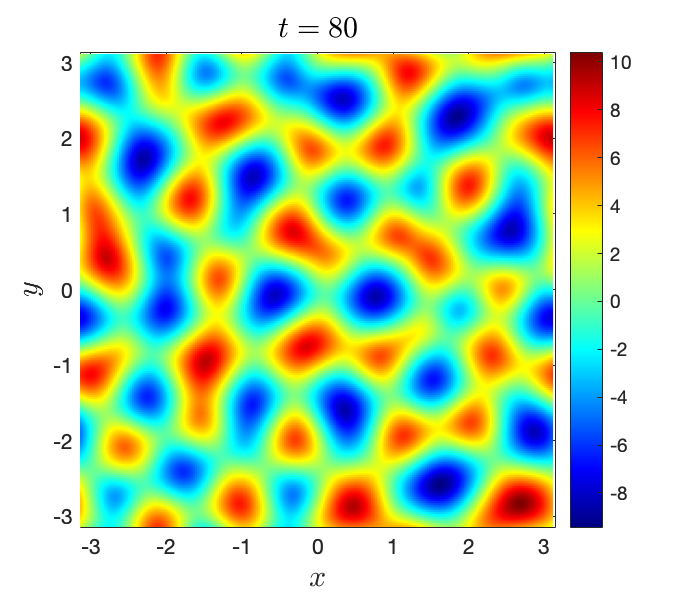}\\
\includegraphics[width=0.32\textwidth]{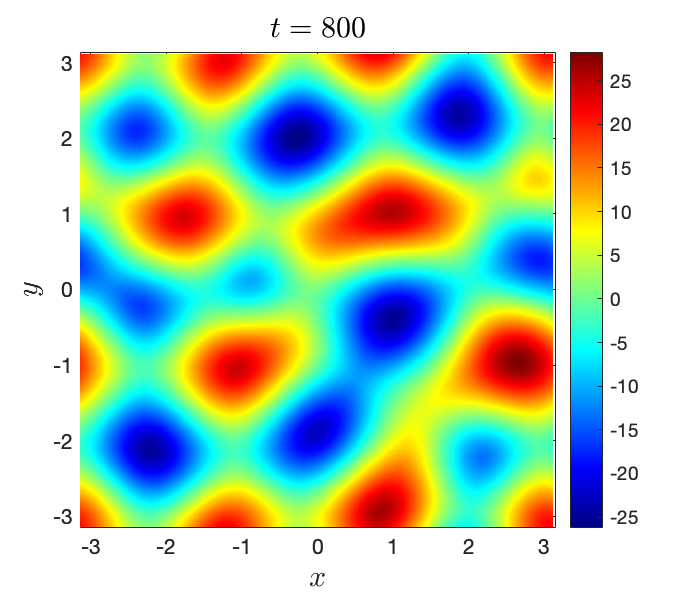}
\includegraphics[width=0.32\textwidth]{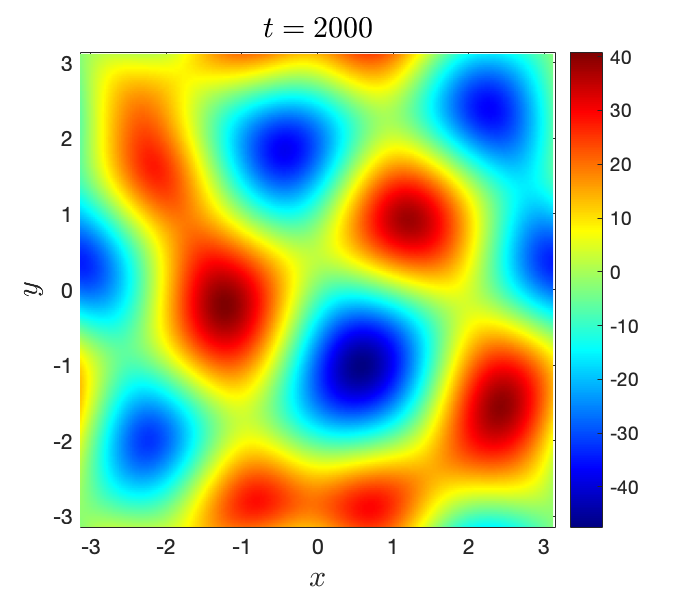}\includegraphics[width=0.32\textwidth]{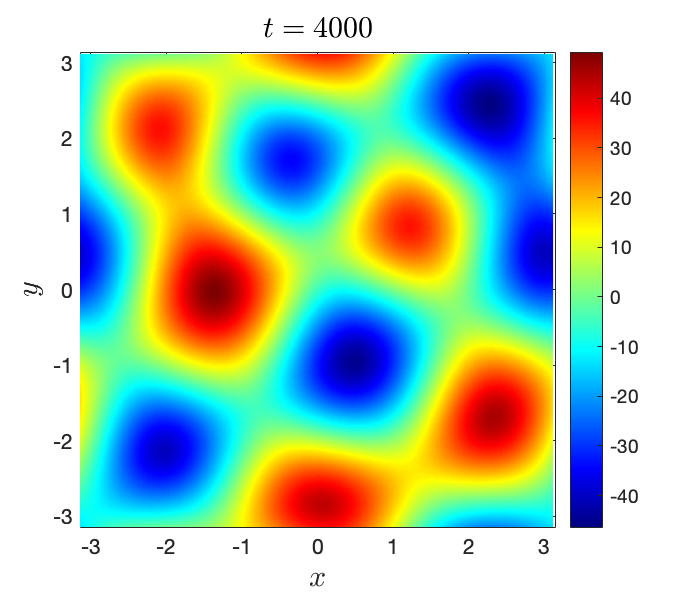}\\
\includegraphics[width=0.32\textwidth]{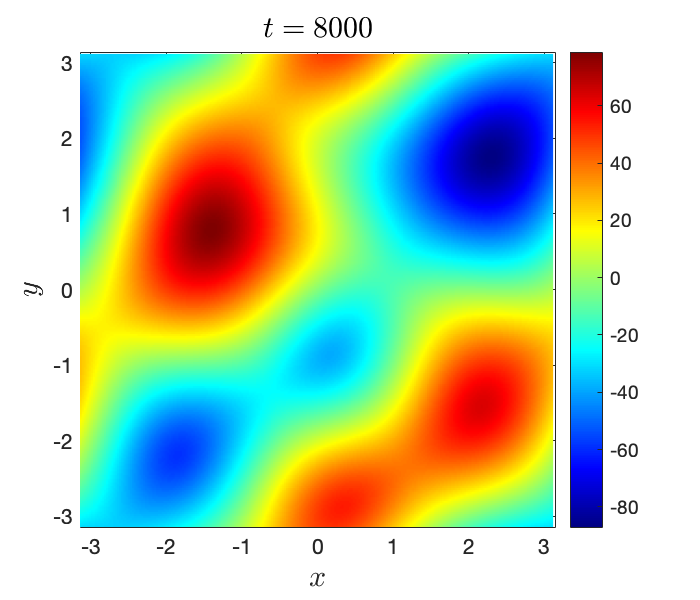}
\includegraphics[width=0.32\textwidth]{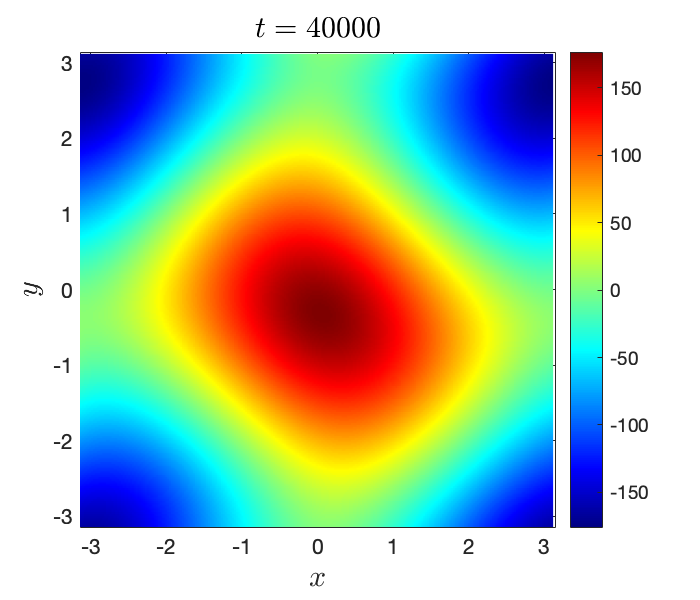}
\includegraphics[width=0.32\textwidth]{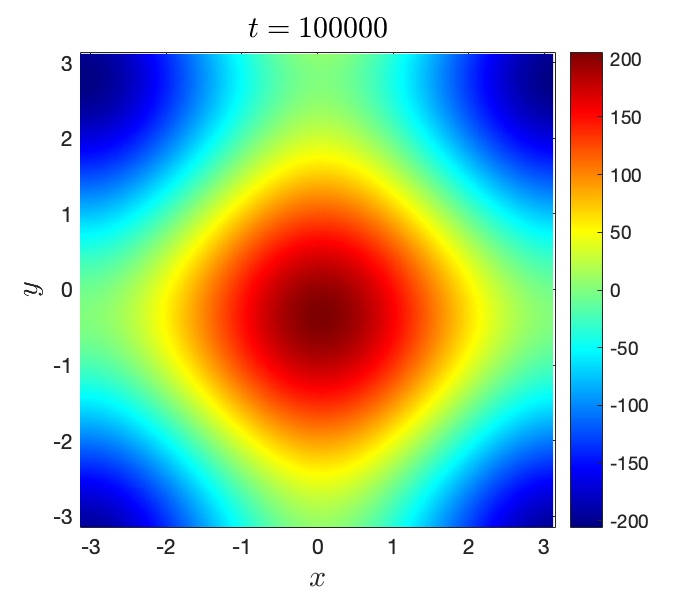}
\caption{\small Snapshots of the solution $h$ to the thin film epitaxy model \eqref{1.1} with no slope selection, computed by the BDF3/EP3 scheme \eqref{eq:sch2} with $\eta = 0.01,~\tau= 0.1,~N_x=N_y = 256$.
 }\label{fig:nss_phase}
\end{figure}

\begin{figure}[!]
\centering
\includegraphics[width=0.5\textwidth]{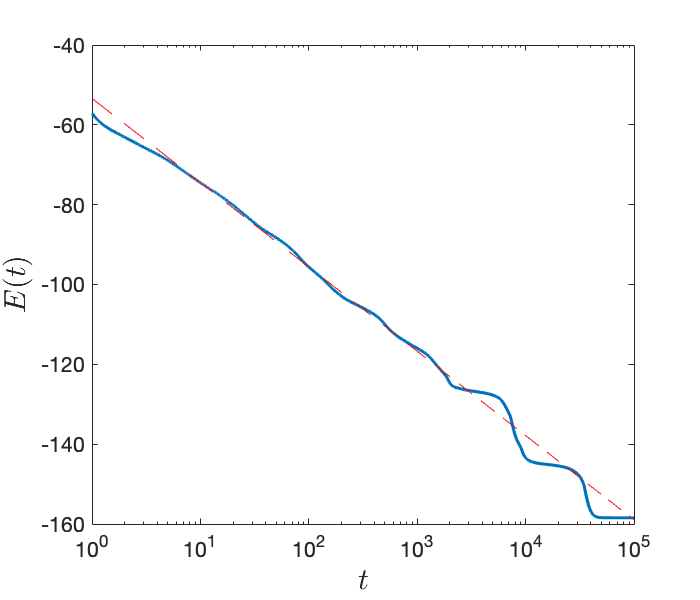}
\caption{\small  Semilog plot (blue curve) of the energy $E$ w.r.t. $t$, computed with $\eta = 0.01,~\tau= 0.1,~N_x=N_y = 256$.
The straight dashed line represents the fitting curve $a \log(t) + b$ with $a = -9.1691, ~b = -53.3853$. 
This fitting only uses the data when $1\leq t\leq 400$.
 }\label{fig:nss_E}
\end{figure}

\begin{figure}[!]
\centering
\includegraphics[width=0.5\textwidth]{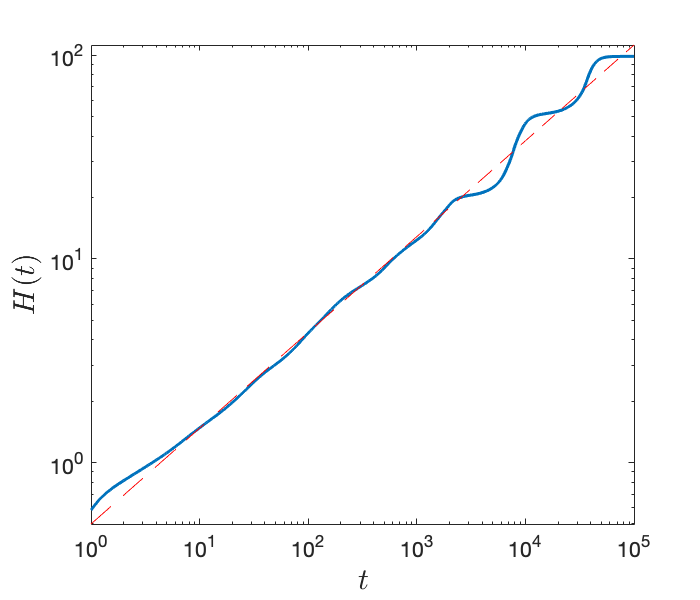}
\caption{\small  Log-log plot (blue curve) of the characteristic height $H$ w.r.t. $t$, computed with $\eta = 0.01,~\tau= 0.1,~N_x=N_y = 256$.
The straight dashed line represents the fitting curve $a t^b$ with $a = 0.4994,~b = 0.4703$.  This fitting only uses the data when $1\leq t\leq 400$.}\label{fig:nss_H}
\end{figure}

\begin{figure}[!]
\centering
\includegraphics[width=0.5\textwidth]{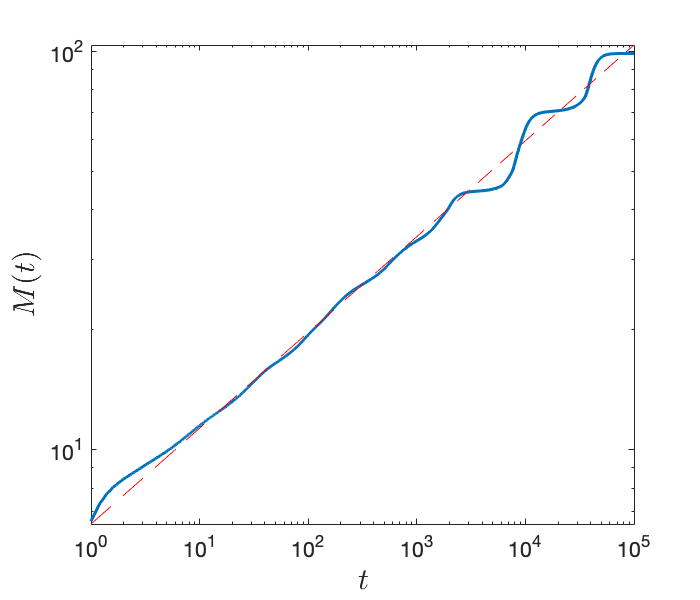}
\caption{\small  Log-log plot (blue curve) of the characteristic slope $M$ w.r.t. $t$, computed with $\eta = 0.01,~\tau= 0.1,~N_x=N_y = 256$.
The straight dashed line represents the fitting curve $a t^b$ with $a = 6.4993,~b = 0.2405$. This fitting only uses the data when $1\leq t\leq 400$. }\label{fig:nss_M}
\end{figure}

{
We also test the unconditional energy boundedness when the time step gets large. As an example we take $\tau=10$ and plot the corresponding energy evolution in Figure \ref{fig:nss_largestep}.
Clearly the energy remains uniformly bounded albeit there is no strict energy dissipation.
}
\begin{figure}[!]
\centering
\includegraphics[width=0.5\textwidth]{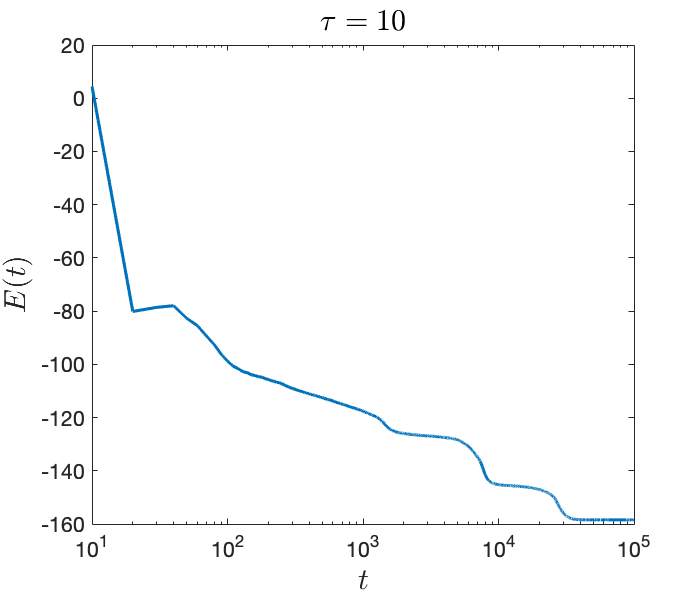}
\caption{\small   Semilog plot of the original energy $E$ w.r.t. time computed with large time step $\tau= 10$ and with other settings the same as Figure \ref{fig:nss_E}. }\label{fig:nss_largestep}
\end{figure}

\section{Concluding remarks}
In this work we considered the classic MBE model with no slope selection.   We use BDF3 for temporal discretization and implicit treatment for the surface diffusion term. The nonlinearity is approximated 
 by an explicit EP3 method. For this BDF3/EP3 method we identified explicit time step constraints 
 and rigorously proved the modified energy dissipation law. Furthermore we introduced 
 a new theoretical framework and showed that the $H^2$-norm of the numerical solutions are unconditionally uniformly bounded, i.e.  the obtained upper bound is independent of the time step. 
  We developed a novel framework for the error analysis for high order methods.  
  To our best knowledge, these kind of results are the first in the literature, albeit for
   a restrictive class of phase field models whose nonlinearity has bounded derivatives.
  We also carried out several numerical
 experiments which show good accordance with theoretical predictions.
It is expected that our new theoretical framework
 can be generalized to many other phase-field models with {\emph{benign (i.e. Lipschitzly bounded)}} nonlinearities. 
 
 \vspace{0.1cm}
 
{\bf Acknowledgement.}
The research of W. Yang is supported by NSFC Grants 11801550 and 11871470.
The work of C. Quan is supported by NSFC Grant 11901281, the Guangdong Basic and Applied Basic Research Foundation (2020A1515010336), and the Stable Support Plan Program of Shenzhen Natural Science Fund (Program Contract No. 20200925160747003).

\frenchspacing
\bibliographystyle{plain}


\end{document}